\documentclass[11pt]{amsart}
\usepackage{amsmath}
\usepackage{cases}
\usepackage{amssymb, verbatim}
\usepackage{comment}
\usepackage{xcolor}

\title[Hermitian-Yang-Mills connections on collapsing elliptically fibered $K3$ surfaces ]{Hermitian-Yang-Mills connections on collapsing elliptically fibered $K3$ surfaces}

\author[V. Datar] {Ved Datar}
 \address{Department of Mathematics, Indian Institute of Science, Bengaluru, India}

 \email{vvdatar@iisc.ac.in}
\author[A. Jacob]{Adam Jacob}
\address{Department of Mathematics, UC Davis,  One Shields Ave, Davis, CA}
 \email{ajacob@math.ucdavis.edu}

\theoremstyle{plain}
\newtheorem{thm}{Theorem}[section]
\newtheorem{prop}[thm]{Proposition}
\newtheorem{defn}[thm]{Definition}
\newtheorem{lem}[thm]{Lemma}
\newtheorem{cor}[thm]{Corollary}

\theoremstyle{defn}

\numberwithin{equation}{section}

\renewcommand{\thanks}[1]{\footnote{#1}} 

\newcommand{\be}{\begin{equation}}
\newcommand{\bea}{\begin{eqnarray}}
\newcommand{\eea}{\end{eqnarray}} \newcommand{\ee}{\end{equation}}
 
 \def\ba{\begin{eqnarray}}
\def\ea{\end{eqnarray}}

\def\cL{{\mathcal L}}

\def\cS{{\mathcal S}}

\def\ra{\rightarrow}

\def\al{\alpha}
\def\b{\beta}

\def\ti{\tilde}

\def\i{\infty}

\def\ra{\rightarrow}

\def\cO{{\mathcal O}}

\def\cO{{\mathcal O}}
\def\cI{{\mathcal I}}

\def\cW{{\mathcal W}}

\def\i{{\mathbf{i}}}

\def\[{{\bf [}}
\def\]{{\bf ]}}

\def\pl{\partial}

\begin{document}
\begin{abstract}

Let $X\rightarrow {\mathbb P}^1$ be an elliptically fibered $K3$ surface, admitting a sequence $\omega_{i}$ of Ricci-flat metrics collapsing the fibers. Let $ V$ be a holomorphic $SU(n)$ bundle over $X$, stable with respect to $\omega_i$. Given the corresponding sequence $\Xi_i$ of Hermitian-Yang-Mills connections on $ V$, we prove that, if $E$ is a generic fiber, the restricted sequence $\Xi_i|_{E}$  converges to a flat connection $A_0$. Furthermore, if the restriction $V|_E$ is of the form $\oplus_{j=1}^n\mathcal O_E(q_j-0)$ for $n$ distinct points $q_j\in E$, then these points uniquely determine $A_0$.

\end{abstract}

\maketitle

\begin{normalsize}

\section{Introduction}

In this paper we study degenerations of Hermitian-Yang-Mills connections on a $K3$ surface. We are motivated by the work of Gross-Wilson \cite{GW}, and later  Gross-Tosatti-Zhang \cite{GTZ, GTZ2}, who study  Ricci-flat metrics on  elliptically fibered Calabi-Yau's as the volume of the fibers tends to zero (see also \cite{HT, To1, To2, TZ}). These types of degenerations relate to the conjectural picture of mirror symmetry put forth by Strominger-Yau-Zaslow \cite{SYZ}, who postulate that mirror Calabi-Yau manifolds are given by dual  torus fibrations over a real base with a singular affine structure. One major challenge when confronting this conjecture is the difficulty associated with constructing Lagrangian torus fibrations on a given Calabi-Yau. However, if one instead considers degenerations of Calabi-Yau's, the fibration structure often becomes apparent in the limit. 

Vafa's extension of the mirror symmetry conjecture to include holomorphic bundles raises the question of how Yang-Mills connections behave under these degenerations  \cite{V}. 
For a general abelian fibered Calabi-Yau, Fukaya makes the following conjecture: Given a sequence of Yang-Mills connections on a family of Calabi-Yau metrics with collapsing fibers, there exists a rectifiable set in the base of real codimension at least 2, such that on all fibers away from this set, the connections have bounded curvature, and  the restriction to each torus fiber converges to a flat connection \cite[Conjecture 5.5]{Fuk1}. In the context of SYZ mirror symmetry, a fiberwise flat connection will define a Lagrangian submanifold $\cL$ in the mirror Calabi-Yau,  and Fukaya further conjectures there exists a corresponding mirror sequence of Lagrangians converging to $\cL$. In this paper we partially address the vector bundle portion of Fukaya's conjecture, in the case of a fixed holomorphic $SU(n)$ bundle over a $K3$ surface.

Our setup is as follows. Let $\pi:X\rightarrow {\mathbb P}^1$ be an elliptic $K3$ surface. Let $\omega_{\mathbb P^1}$ be  a K\"ahler form on ${\mathbb P}^1$, and $\omega_X$ a K\"ahler form on $X$. Consider the family $\pi^*\omega_{\mathbb P^1}+t^2 \omega_X$, which approaches the boundary of the K\"ahler cone as $t\rightarrow 0$, and let $\omega_t$ be the unique Ricci-flat metric in the class $[\pi^*\omega_{\mathbb P^1}+t^2 \omega_X]$ given by Yau's theorem \cite{Yau}.
Next, let $( V,\bar\partial_{\Xi })$ be a holomorphic $SU(n)$ bundle over $X$, with a fixed  metric $H_0$. Assume there exists a sequence $t_i\rightarrow 0$ such that the bundle $V$ is stable with respect to $\omega_{t_i}$. By the theorem of  Donaldson, Uhlenback-Yau \cite{Don1, UY}, there exists a corresponding sequence of connections $\Xi_i$ solving  the Hermitian-Yang-Mills (HYM) equations: 
\be
F_{\Xi_i}\wedge \omega_{t_i}=0\qquad{\rm and}\qquad (F_{\Xi_i})^{0,2}=0.\nonumber
\ee
Furthermore, each $\Xi_i$ is complex gauge equivalent to $\Xi $, so they define the same holomorphic structure (see \eqref{complexifiedgauge}). We now state our main result:

 \begin{thm}
 \label{maintheorem}With the set-up as above:
\begin{enumerate}
\item There exists a finite subset $Z\subset \mathbb{P}^1$, such that for any $x\in \mathbb{P}^1\setminus Z$, if $E = \pi^{-1}(x)$, then the restriction $\Xi_i\big|_{E}$ converges smoothly, along a subsequence and modulo unitary gauge transformations, to a flat connection on the fiber.

\item Furthermore, if the restriction $V|_E$ is isomorphic to a direct sum of line bundles $\oplus_{j=1}^n\mathcal O_E(q_j-0)$ for $n$ distinct points $q_j\in E$, then the limiting flat connection is uniquely determined, and given by
 \be
 \label{limit}
 A_0=\frac{\pi}{{\rm Im}(\tau)} \left({\rm diag}\{\bar q_1,...,\bar q_n\}dz-{\rm diag}\{q_1,...,q_n\}d\bar z\right),
 \ee
 where $z$ is the holomorphic coordinate on the fiber, and $\tau$ determines the complex structure. In this case we also have the following convergence:
 \be
 ||\Xi_i|_E-A_0||_{L^2_1(E,H_0,g_0,A_0)}\rightarrow 0.\nonumber
 \ee
Here $g_0$ is a flat reference metric on $E$, and the flat connection $A_0$ is used to compute derivatives. 
\end{enumerate}
  \end{thm}

Note that in the second point above, no gauge transformations are needed for convergence. The first point follows from a bubbling argument. Our sequence of connections  has bounded Yang-Mills energy, thus there can only be a finite number of bubbles, and we show away from these bubbles the curvature of  $\Xi_i|_{E}$ must approach zero in the $C^0$ norm. This step closely follows  two   cases from   Dostoglou-Salamon, from their proof of the Atiyah-Floer conjecture \cite{DS, DS2}. The key difference here is that our Ricci flat metric $\omega_{t_i}$ is not a product metric, so we rely on certain convergence results for $\omega_{t_i}$. 

It then follows that $\Xi_i|_{E}$  must  approach some limiting flat connection, and the  main contribution of this paper is the explicit identification of the limit, under the assumption that  the restriction of the holomorphic bundle $V|_E$ is isomorphic to a direct sum of line bundles $\oplus_{j=1}^n\mathcal O_E(q_j-0)$. We   use the observation that, because our sequence of connections $\Xi_i$ all define a fixed holomorphic structure,  there exists a sequence of Hermitian endomorphisms $s_i$ satisfying $\Xi_i|_E=e^{s_i}( A_0)$ (where this action is defined by \eqref{complexifiedgauge}). Although there is no hope of achieving $C^0$ control of $s_i$, we prove a gauge fixing result, and demonstrate that there exists a suitable normalization $s'_i$, defining the same connection, which in addition satisfies a uniform $C^0$ bound. This  significant step  is detailed in Theorem \ref{gaugef}, which in particular hinges on a Poincar\'e inequality \eqref{awesome1}, where the explicit form of $V|_E$ is used. From here, convergence of $\Xi_i|_E$ to $A_0$ stated in Theorem \ref{maintheorem} follows from standard theory.

Next we turn to a specific  geometric setup where Theorem \ref{maintheorem} applies. Although this setup requires more assumptions, it has the benefit of producing explicit examples of bundles where the the limiting flat connection can be identified on a generic fiber. We now assume $\pi: X\rightarrow {\mathbb P}^1$ is a projective, elliptic $K3$ surface with a section $\sigma$ and singular fibers of type $I_1$ or $II$. Assume the restriction of $V$ to a generic fiber is semi-stable and regular (see Section \ref{bundles} for relevant defintions).  Then by the work of  Friedman-Morgan-Witten in \cite{FMW2},   there exists a divisor  $D_V\in |n\sigma(\mathbb P^1)+kl|$,
called the {\it spectral cover} associated to $V$, where $l$ denotes the effective divisor class of the fibes, and  $k\in\mathbb Z$ satisfies $0\leq k\leq c_2(V)$.  If $D_V$ is reduced and irreducible, then    $V$ is stable with respect to $\pi^*[\omega_{\mathbb P^1}]+t^2[\omega_X]$ for any ample class $[\omega_X]$ on $X$,  for  $0<t^2\leq(\frac{n^3}{4}c_2(V))^{-1}.$ Thus, for any sequence $t_i\rightarrow0$  we can always find a corresponding sequence of HYM connections $\Xi_i$ on $V$. More importantly, the intersection of the spectral cover $D_V$ with a generic fiber $E$ precisely picks out the points ${q_1,...,q_n}$ from \eqref{limit}, and so the limiting flat connection is uniquely determined away from the ramification points of $D_V$. 
\begin{cor}
\label{maincor}
Assume $\pi: X\rightarrow {\mathbb P}^1$ is a projective, elliptic $K3$ surface with a section, with singular fibers of type $I_1$ or $II$. Let $V$ be a holomorphic $SU(n)$ bundle where the restriction to a generic fiber is semi-stable and regular,  and assume the spectral cover $D_V$ is reduced and irreducible.  Then for any sequence  $t_i\rightarrow 0$, there exists  a sequence of HYM connections $\Xi_i$   on $V$ corresponding to $\omega_{t_i}$. Furthermore, away from a finite number of fibers, there exists a HYM connection $\Xi_0$ uniquely determined by $D_V$, satisfying $\Xi_0|_{E}=A_0$, where $A_0$ is defined via \eqref{limit}. Specifically, on a generic fiber $E$ the points ${q_1,...,q_n}$ defining $A_0$  are given by
\be
D_V\cap E=q_1+\cdots+q_n.\nonumber
\ee
On $E$  we again have the convergence
 \be
 ||\Xi_i|_E-\Xi_0|_E||_{L^2_1(E,H_0,g_0, A_0)}\rightarrow 0.\nonumber
 \ee
\end{cor}
The above setup is particularly attractive in that  it allows for us to specify the limiting flat connection in a family that varies homomorphically in the base. Thus our result   is a natural starting place to explore convergence in general, as opposed to only in the fiber direction.

Since adiabatic limits of Yang-Mills connections are fairly well studied, we now put our work in the  context of previous results. Working on the product of two  compact Riemann surfaces with trivial fibration $\pi:\Sigma_1\times\Sigma_2\rightarrow\Sigma_1$, J. Chen considers a family of metrics collapsing the fibers, and analyzes the convergence of a corresponding family of anti-self dual Yang-Mills connections \cite[Theorem 4.10]{Chen}. Assuming that $\Sigma_2$ has genus at least two, he proves that, modulo a sequence of gauge transformations and away from a bubbling set, the fiber component of the connection will converge continuously to a flat connection. Following this work, and using Fukaya's gauge fixing theorem  \cite[Theorem 1.7]{Fuk2}, T. Nishinou improves upon Chen's result, demonstrating smooth convergence away from a   finite number of fibers and modulo gauge transformations, where the restriction of the limit  to a fiber will be flat \cite[Theorem 1.2]{N1}. This result requires  the moduli space of flat connections over $\Sigma_2$ to be  smooth and of expected dimension, with no reducible flat connections. The failure of such an assumption to hold over an elliptic curve is a major obstacle to extending the above results  to elliptic fibrations. 

In the case of  $SU(2)$ bundles over the product of elliptic curves, in \cite{N3}  Nishinou is able to partially extend his above results, after utilizing his gauge fixing theorem from \cite[Theorem 3.11]{N2}. In some ways, our Theorem \ref{gaugef} can be thought of as a generalization to higher rank of this gauge fixing theorem, although   we have already assumed existence of a fixed holomorphic structure. In fact, this assumption serves as a major simplification throughout our paper, compared to the general case of a sequence of anti-self dual Yang-Mills connections considered in \cite{Chen, DS, Fuk2, N1,N3}. The most notable simplification is that, because our sequence of connections $\Xi_i$ are all complex gauge equivalent to $\Xi $, we can bypass working with a sequence of holomorphic maps (which plays a role in \cite{Chen, N1, N3}), as well as the more difficult type three bubbles of Dostoglou-Salamon \cite{DS, DS2}. Additionally, this assumption allows us to prove the convergence  in the second  point of Theorem \ref{maintheorem} directly, without relying on unitary gauge transformations.

Although our main result only applies to the restriction of $\Xi_i$ to each fiber, one may hope to demonstrate convergence on any compact set away from a finite number of fibers. Nishinou achieves this in \cite{N1} and \cite{N3}, as his assumptions allow a Poincar\'e type inequality in a neighborhood of a fiber, even in the elliptic curve case (this follows from the estimate in Lemma 6.43 from \cite{Fuk2}). This estimate implies that once  $C^0$ control of the complexified gauge transformation is demonstrated on one fiber, it holds for nearby fibers. Unfortunately we are unable to extend Lemma 6.43 from \cite{Fuk2} to our setting, as our Poincare inequality (Proposition \ref{poincare}) requires a normalization that only holds fiberwise. Another result in this direction is proven by Fu in \cite{Fu}, who considers a specific rank two bundle over the product of  two elliptic curves which is given by a two sheeted spectral cover. He defines a reference metric which satisfies desired asymptotic behavior near the ramification points of the cover, and then demonstrates that the a sequence of HYM metrics will converge smoothly to this reference metric.  Because in our setting we consider a spectral cover as well,   one may hope to extend Fu's result to the K3 surface case. Here, one major difficulty is the problem of constructing a reference metric near the singularities of the fibration. It is possible that the asymptotics of  the metrics constructed in \cite{CJY2, Simp2, Simp3} may provide a clue, and we hope to investigate these types of constructions in future work.

Finally, we remark that after an earlier draft of this paper appeared, building on this work, the authors, along with Y. Zhang,  were able to demonstrate convergence of $\Xi_i$ on compact sets away from a finite number of fibers, under the assumptions of Corollary \ref{maincor}.  We direct the reader to \cite{DJZ} for details.

Our paper is organized as follows.  In Section \ref{semiflat} we describe in detail semi-flat K\"ahler metrics on a K3 surface, which serve as a local model for our degenerating Ricci-flat metrics away from the singular fibers. Next in Section \ref{bundles} we introduce the necessary background on holomorphic vector bundles over elliptic fibrations,  and state some preliminary results. Our bubbling argument in described Section \ref{bubs}. We then turn to identifying the limiting flat connection, and prove our gauge fixing result is  Section \ref{gaugefixingsection}. In Section \ref{convo} we complete the proof of our main theorem, and demonstrate convergence of our connections.

\vspace{\baselineskip}
\noindent
{\bf Acknowledgements} We would like to thank T. Nishinou, for  providing  a revised draft of \cite{N3}, and P. Feehan for introducing us to reference \cite{F}. We also thank the referee for many valuable suggestions. V. Datar is partially supported by NSF RTG grant DMS-1344991 and A. Jacob is partially supported by a grant from the Hellman Foundation, and a Simons collaboration grant.

\section{Semi-flat K\"ahler metrics}
\label{semiflat}

In this section, we review   the construction of semi-flat K\"ahler metrics on a K3 surface. These metrics will not only  describe the limiting behavior of the Ricci-flat metrics  in dilated coordinates, and thus play a role in our bubbling argument, but they will also be useful for our understanding of the holomorphic structure   of $ V$. To begin,  we  introduce the notion of a special K\"ahler metric, which lives on the base of  our elliptic fibration, and are a useful starting place to defining the semi-flat metric.  We will closely follow the paper of Freed \cite{Freed}.

Let $B$ be a Riemannian manifold of real dimension two.  Assume $TB$ admits a flat, torsion free connection $\nabla^B$, which gives a covering of $B$ by local affine coordinate charts. Furthermore assume the coordinate transformations lie in $SL(2,{\mathbb R})$. Let $(x^1,x^2)$ be coordinates in a local chart, and let $ \phi_{ij}dx^idx^j$ be a Hessian metric solving the real Monge-Ampere equation
\be
\label{MA}
{\rm det}(\phi_{ij})=1.
\ee
Here we use the notation $ \phi_{ij}:=\frac{\partial^2\phi}{\partial x^i\partial x^j}$ for a smooth function $\phi$ on $B$. We also denote $\frac{\partial \phi}{\partial x^i}$ by $\phi_i$.

Consider the locally defined 2-form $\omega_B= dx^1\wedge dx^2.$ Because $ SL(2,{\mathbb R})= Sp(2,{\mathbb R})$,  any matrix $A\in SL(2,{\mathbb R})$ preserves $\omega_B$, and thus $\omega_B$ is well defined on all of $B$. It defines both a natural symplectic form and a volume form. Furthermore, $\nabla^B\omega_B=0$, and so $\nabla^B$ is a symplectic connection.  Taken together, $\omega_B$ and $\phi_{ij}$ define an almost complex structure $I$ on $TB$, which in coordinates can be expressed by
\be
I\left(\frac{\pl}{\pl x^k}\right)=-\phi_{k2}\frac{\pl}{\pl x^1}+\phi_{k1}\frac{\pl}{\pl x^2},\nonumber
\ee
for $k=1,2$. One can show explicitly that Nijenhuis tensor of $I$ vanishes and thus it is integrable. 
\begin{defn}
 $(B,\omega_B, I)$ is special K\"ahler if it admits a real, flat, torsion-free, symplectic connection $\nabla^B$ satisfying
\be
d_{\nabla^B} I=0.\nonumber
\ee
\end{defn}
To see that $B$ is special K\"ahler, note in affine coordinates the flat connection $\nabla^B$ is simply given by $d$, and so
\be
\frac{\pl}{\pl x^k}I^p{}_q -\frac{\pl}{\pl x^q}I^p{}_k =0,\nonumber
\ee
which follows because $\phi_{ij}$ is a Hessian metric.

Given the complex structure $I$, we can give holomorphic coordinate functions on $B$. 
\begin{lem}
\label{holosemiflat}
The functions \be
w=x^1+\i\phi_2\qquad{\rm and}\qquad \xi=-x^2+\i\phi_1.\nonumber
\ee
are holomorphic with respect to the complex structure $I$. 
\end{lem}
\begin{proof}
Taking the exterior derivative gives $dw=(1+\i\phi_{12})dx^1+\i\phi_{22}dx^2$ and $d\xi=\i \phi_{11}dx^1-(1-\i\phi_{12})dx^2$. By our explicit representation of $I$ it is easy to check that $I(dw)=\i dw$ and $I(d\xi)=\i d\xi$.
\end{proof}
For the remainder of the paper we choose $w$ as our holomorphic coordinate on the base. We would like to better understand the holomorphic vector field $\frac{\pl}{\pl w}$. First, note that the coordinate transformation $T(x^1, x^2)\ra(w, \bar w)$ has pushforward matrix
\be
T_*=\left( 
\begin{array}{cc}
1+\i\phi_{12} &\i\phi_{22}\\
1-\i\phi_{12}&-\i\phi_{22}\end{array} 
\right)\qquad{\rm and}\qquad T_*^{-1}=\frac12\left( 
\begin{array}{cc}
1 &1\\
-\i\frac{\phi_{11}}{1+\i\phi_{12}}&\i\frac{\phi_{11}}{1-\i\phi_{12}}\end{array} 
\right),\nonumber\ee
where we have used \eqref{MA}. We now compute the partial derivative 
\bea
\label{taudef}
\frac{\pl \xi}{\pl w}&=&\frac{\pl \xi}{\pl x^1}\frac{\pl x^1}{\pl w}+\frac{\pl \xi}{\pl x^2}\frac{\pl x^2}{\pl w}\nonumber\\
&=&(\i\phi_{11})\left(\frac12\right)+(-1+\i\phi_{12})\left(\frac{-\i}2\frac{\phi_{11}}{1+\i\phi_{12}}\right)\nonumber\\
&=&\frac\i2\phi_{11}\left(1+\frac{1-\i\phi_{12}}{1+\i\phi_{12}}\right)=\i\frac{\phi_{11}}{1+\i\phi_{12}}=\i\frac{1-\i\phi_{12}}{\phi_{22}},
\eea
where the last equality follows from \eqref{MA}.  For simplicity we will use the notation $\tau:=\frac{\pl \xi}{\pl w},$
 as $\tau$ will define the complex structure of our elliptic fibers. Then, by the explicit formula from Freed \cite[Equation (1.12)]{Freed}, we see
\be
\frac{\pl}{\pl w}=\frac12\left(\frac{\pl}{\pl x^1}-\tau\frac{\pl}{\pl x^2}\right).\nonumber
\ee

Next we construct a hyper-K\"ahler structure on $TB$. Quotienting out $TB$ by a lattice $\Lambda$ will give a local model for our elliptic fibration $X$ away from the singular fibers. 
To begin, consider the following  extension of  our Hessian metric to $TB$:
 \be
g=\phi_{ij}\left({dx^idx^j}+dy^i dy^j\right).\nonumber
\ee
With this metric we define three complex structures which make up a hyper-K\"ahler triple.

By a slight abuse of notation, let $I$ denote the complex structure on $TB$ induced from the complex structure $I$ on the base. In particular $I$ can be expressed as
\be
I\left(\frac{\pl}{\pl x^k}\right)=-\phi_{k2}\frac{\pl}{\pl x^1}+\phi_{k1}\frac{\pl}{\pl x^2}\qquad{\rm and}{\qquad }I\left(\frac{\pl}{\pl y^k}\right)=\phi_{k2}\frac{\pl}{\pl y^1}-\phi_{k1}\frac{\pl}{\pl y^2}\nonumber
\ee
for $k=1,2$. The corresponding K\"ahler form is given by
\be
\label{omega_I}
\omega_{I}= {dx^1\wedge dx^2}-dy^1\wedge dy^2.
\ee
In this complex structure the fibers are holomorphic subvarieties.

Next we consider a  complex structures $J $ where the fibers are special Lagrangian, defined by
\be
J \left(\frac{\pl}{\pl x^k}\right)= \frac{\pl}{\pl y^k}\qquad{\rm and}{\qquad }J \left(\frac{\pl}{\pl y^k}\right)= \frac{\pl}{\pl x^k}.\nonumber
\ee
Using the metric $g $ the corresponding K\"ahler form is
\be
\omega_J=\phi_{ij} dx^i\wedge dy^j.\nonumber
\ee

Finally, one can define  the complex structure  $K=JI$, which together with $g $ gives the K\"ahler form
\be
\omega_K=dx^1\wedge dy^2-dx^2\wedge dy^1.\nonumber
\ee
It is easy to see that $\omega_{I }$, $\omega_J$ and $\omega_K$ are closed, and by a lemma of Hitchin \cite{H} it follows that $I$, $J$, and $K$ are integrable complex structures. In the standard fashion we can construct the following top dimensional holomorphic forms:
\be
\Omega_I=\omega_K+\i\,\omega_J\nonumber
\ee
\be
\Omega_{J}=\omega_{I}+\i\,\omega_K\nonumber
\ee
\be
\Omega_{K}=\omega_{I}+\i\,\omega_J,\nonumber
\ee
giving hyper-K\"aher triple. Note the metric and complex structures defined above are invariant under translation in the $y-$coordinates. Thus if $\Lambda$ is the standard lattice $\langle 1,1\rangle$, the entire setup will descend to the elliptically fibered manifold $TB/\Lambda$.

We now construct complex coordinates on  $TB/\Lambda$. We have a complex coordinate $w$ on the base $B$, and in the fiber direction we define
\be
\label{complexcoord}
z=\tau y^1+y^2.
\ee
Here $\tau$ represents the complex period of the elliptic curve, and is defined above in \eqref{taudef}. By definition $\tau:=\frac{\pl \xi}{\pl w}$, and since $\xi$ is a holomorphic function, $\tau$ is holomorphic in $w$ as well. This leads to the following:
\begin{lem}
The coordinates $(w,z)$ are holomorphic coordinates on $TB/\Lambda$ with respect to  $I$.
\end{lem}

\begin{proof}
Lemma \ref{holosemiflat} shows that $w$ is holomorphic, and so it remains to be seen that $dz(V)=0$ for all vector fields $V$ of type $(0,1)$. Taking the exterior derivative of $z$ gives
\be
\label{def z} 
dz=\tau dy^1+y^1\frac{\pl\tau}{\pl w}dw+dy^2,
\ee
where we used $\tau$ is holomorphic. At first glance the term $y^1\frac{\pl\tau}{\pl w}dw$ may seem out of place, however, it is important to remember that unless $\tau$ is constant, our local picture is not the cartesian product of the base with an elliptic curve, and so this term is expected.

Consider the vector field
\be
\frac{\pl}{\pl \bar z}=\frac12\left(\frac{\pl}{\pl y^1}-\frac{\i\phi_{11}}{1+\i\phi_{12}}\frac{\pl}{\pl y^2}\right).\nonumber
\nonumber
\ee
From the explicit form of $dz$, we see that in order for $\frac{\pl}{\pl \bar z}$ to be anti-holomorphic, it needs to be killed by the form $dw$ on the base.

Using the definition of $I$, and the fact that ${\rm det}(\phi_{ij})=1$, we compute
\bea
I\left(\frac{\pl}{\pl\bar z}\right)&=&\frac12\left( \phi_{12}-\frac{\i\phi_{11}\phi_{22}}{1+\i\phi_{12}}\right)\frac{\partial}{\partial y^1}+\frac12\left(-\phi_{11}+\frac{\i\phi_{11}\phi_{12}}{1+\i\phi_{12}} \right)\frac{\partial}{\partial y^2}\nonumber\\
&=&\frac12\left( -\i\right)\frac{\partial}{\partial y^1}+\frac12\left(-\frac{\phi_{11}}{1+\i\phi_{12}} \right)\frac{\partial}{\partial y^2}=-\i\frac{\pl}{\pl\bar z},\nonumber
\eea
which demonstrates this vector field is of type $(0,1)$. By Lemma \ref{holosemiflat} it follows that $dw\left(\frac{\pl}{\pl\bar z}\right)=0$. Additionally, we have
\be
(\tau dy^1+dy^2)\left(\frac{\pl}{\pl\bar z}\right)=\frac12(\tau-\tau)=0.\nonumber
\ee
So $dz(\frac{\pl}{\pl\bar z})=0$. Since $\frac{\pl}{\pl\bar w}$ and $\frac{\pl}{\pl\bar z}$ span all local $(0,1)$ vector fields, we conclude that $(w,z)$ are holomorphic coordinates.
\end{proof}

We conclude this section with a more detailed discussion of how the Ricci-flat K\"ahler metrics  behave in the limit. We recall our setup from the introduction. Let $\pi:X\rightarrow {\mathbb P}^1$ be an elliptic $K3$ surface, and denote by $Z_\pi$   the image of the singular fibers.  Let $\omega_{\mathbb P^1}$ be  the Fubini-Study metric on on ${\mathbb P}^1$, and $\omega_X$ a K\"ahler form on $X$. Let $\omega_t$ be the unique Ricci-flat metric in the class $[\pi^*\omega_{\mathbb P^1}+t^2 \omega_X]$. The convergence we need is local, so we fix a small, simply connected open set $U\subset \mathbb P^1$ away from $Z_\pi$, and define $X_U:=\pi^{-1}(U)$. On $U$ we can consider the K\"ahler form $\omega_B$ along with the Hessian metric $\phi_{ij}$, which on this small open set is equivalent to   $\omega_{\mathbb P^1}$.  On $X_U$ we have the fixed background metric $\omega_X$, but we also have the K\"ahler form $\omega_I$ as defined above, called the {\it semi-flat metric}, and we denote it by $\omega_I=:\omega_{SF}$ for emphasis.

We will need the following uniform equivalence result. By \cite[Lemma 4.1]{GTZ}, there exists a constant $C$ so that for $t$ small enough
\be
\label{equiv}
C^{-1}\left(\pi^*\omega_{\mathbb P^1}+t^2\,\omega_X\right)\leq\omega_t\leq C\left(\pi^*\omega_{\mathbb P^1}+t^2\,\omega_X\right).
\ee
We also need a result that demonstrates  how $\omega_t$ degenerates. Consider the projection $p:U\times \mathbb C\rightarrow (U\times \mathbb C)/\Lambda=:X_U$, and the coordinate transformation $L_t: U\times \mathbb C\rightarrow U\times \mathbb C$ defined by 
\be
\label{fiberscale}
L_t(x,y)=(x,\frac{y}{ t}).
\ee
This coordinate transformation is a dilatation, designed so that the size of the fibers of $\pi$ with respect to the metric $L_t^*p^*\omega_t$ are fixed.  We will use the following result of  Hein-Tosatti from \cite[Proof of Theorem 1.1]{HT}, which demonstrates how the semi-flat metric serves as a model for the limiting behavior of $\omega_t$.

\begin{prop}[Hein-Tosatti \cite{HT}] 
\label{HTbound}
There exists a constant $C$ so that for $t$ small enough
\be
C^{-1}p^*\left(\pi^*\omega_{B}+\omega_{SF}\right)\leq L_t^*p^*\omega_t\leq Cp^*\left(\pi^*\omega_{B}+\omega_{SF}\right).\nonumber
\ee
\end{prop}
This  estimate also appears in \cite{TZ}, with the extra assumption that $X$ is projective. For more details on convergence results, we direct the reader to \cite{GTZ,GTZ2}.

\section{Holomorphic bundles over elliptic manifolds}
\label{bundles}

In this section we provide the necessary background on holomorphic vector bundles, including the relevant notions of stability needed to construct our sequence of HYM connections. We also introduce the construction of a spectral cover associated to $V$, following Friedman-Morgan-Witten \cite{FMW2} (see also \cite{Friedman, FMW, Mo}), and conclude the section with the construction of $\Xi_0$ used in Corollary \ref{maincor}.

To begin, suppose $(X,\omega)$ is a compact K\"ahler manifold of complex dimension $m$. Let $( V,  \bar\partial_{\Xi })$ be a holomorphic bundle over $X$.  For any Hermitian metric $H_0$ on $ V$, there exists a unique connection, called the Chern connection, compatible with both the metric and the holomorphic structure, which we denote by $\Xi$.  The {\it degree} of $ V$ is defined by the following integral:
\be
{\rm deg}( V,\omega)=\i\int_X{\rm Tr}(F_{\Xi })\wedge\omega^{m-1}.\nonumber
\ee
Given two metrics on $ V$, the curvatures of the two corresponding Chern connections will differ by a $\partial\bar\partial$-exact term, demonstrating that the degree   is independent of a choice of metric. Furthermore the degree does not depend on the representative of the K\"ahler class $[\omega]$. However, for $m>1$, changing the class may change the degree.

\begin{defn}
$( V,\omega)$ is {\it stable} if, for all proper, torsion-free subsheaves $\mathcal F\subset E$, 
\be
\frac{{\rm deg}(\mathcal F,\omega)}{{\rm rk}(\mathcal F)}<\frac{{\rm deg}( V,\omega)}{{\rm rk}( V)}.\nonumber
\ee
$( V,\omega)$ is {\it semi-stable} if the above expression holds with a weak inequality. 
\end{defn}
Note that if $\mathcal F$ is not locally free, its degree is defined by computing  the degree of ${\rm det}(\mathcal F)$, which is always a line bundle.

On any complex manifold, the space of one forms decomposes into the eigenspaces for $\pm\i$ with respect to the complex structure. This allows us to write the any connection $\Xi$ as $\Xi=\Xi^{1,0}+\Xi^{0,1}$. Using this decomposition, one can define an action of the complexified gauge group on the space of connections. Specifically, if $\sigma\in GL( V)$, then
\be 
\label{complexifiedgauge}
\sigma (\Xi)=\sigma^*{}^{-1} \Xi^{1,0}  \sigma^*+\sigma^*{}^{-1}\partial \sigma^*+\sigma \Xi^{0,1} \sigma^{-1}-\bar\partial\sigma \sigma^{-1},
\ee
Note that if $\sigma$ is in fact unitary,  the above action reduces to the standard action of the unitary gauge group. In this case, we use the standard notation $u^*\Xi$ for the unitary action.

We now turn to the HYM equations on a general K\"ahler manifold:
\be
\label{HYM1}
\i F_{\Xi}\wedge \omega^{m-1}=\frac{m\,{\rm deg}(V,\omega)}{{\rm rk(V)}{\rm Vol}(X)} Id_V\,{\omega^m}\qquad{\rm and}\qquad (F_{\Xi})^{0,2}=0.
\ee
For any metric $g$ on $X$ and connection $\Xi$ on $ V$, the Yang-Mills energy is defined by the following integral
$$\mathcal{YM}(\Xi ,g) := \int_X|F_{\Xi }|^2_{H_0, g}\,dV_g. $$ 
Critical points of this energy functional are called Yang-Mills connections, and one can check using the K\"ahler identities that HYM connections are are special class of Yang-Mills connections which are compatible with the complex structure on $X$.

Note that second equation in \eqref{HYM1} stipulates that $\Xi$ is compatible with the holomorphic structure on $ V$. By definition this second equation is satisfied by the Chern connection $\Xi $, and one can check this compatibility  is preserved along the action \eqref{complexifiedgauge}. This leads to the following question: Given a holomorphic vector bundle with fixed metric $H_0$, does there exist a solution $\hat \Xi $ to  \eqref{HYM1} in the orbit of  \eqref{complexifiedgauge}? A definitive answer to this question was given by Donaldson, Uhlenbeck-Yau, in the following fundamental result.

\begin{thm}[Donaldson \cite{Don1}, Uhlenbeck-Yau \cite{UY}]
A holomorphic bundle $ V$ over $(X,\omega)$ admits a unique Hermitian-Yang-Mills connection in the complex gauge orbit of the Chern connection
if and only if it is stable.
\end{thm}
In fact, one can prove that if $\hat \Xi $ is the unique Hermitian-Yang-Mills connection, it can be expressed as $\hat \Xi =e^s (\Xi )$, where $s$ is a trace free Hermitian endomorphism of $ V$.

Given this background, we return to our   setup. Let $\pi:X\rightarrow\mathbb P^1$ be an elliptically fibered $K3$ surface, and  let $\omega_{t}$ be the unique Ricci-flat K\"ahler metric in the class $[\pi^*\omega_{\mathbb P^1}+t^2\,\omega_X]$. Assume $( V,\bar\partial_{\Xi})$ is a holomorphic $SU(n)$ bundle over $X$. This implies the curvature $F_{\Xi}$ is  trace free, and so ${\rm deg}( V, \omega_{t})=0$ for all $t$. Furthermore, assume that $( V,\bar\partial_{\Xi})$ is stable with respect to $\omega_{t_i}$ for some sequence $t_i\rightarrow 0$. Let $ g_{t_i}$ be the K\"ahler metrics associated to  $\omega_{t_i}$. By the theorem of Donaldson-Uhlenbeck-Yau there exists a corresponding sequence of HYM connections $\Xi_i$ on $ V$. Note that in our particular setting, the HYM equations take the simpler form
\be
F_{\Xi}\wedge \omega=0\qquad{\rm and}\qquad (F_{\Xi})^{0,2}=0.\nonumber
\ee

We will need the following Lemma, which states that the Yang-Mills energy of a HYM connection is a topological invariant. This result is standard, and can be found, for instance, in \cite{DK}. We include a proof for the reader's convenience. 
\begin{lem}
\label{YMcontrol}
The Yang-Mills energy of $\Xi_i$ with respect to the metric $g_{t_i}$, is fixed, i.e.
\be
\mathcal{YM}(\Xi_i,g_{t_i})=\mathcal{YM}(\Xi, g_{t_0}).\nonumber
\ee 
\end{lem}
\begin{proof}
Let $\i\Lambda_\omega$ denote the adjoint of wedging with the K\"ahler form $\omega$. Then the equation $ F_\Xi\wedge \omega=0 $ can be equivalently expressed as $\i\Lambda_\omega F_\Xi=0$. Equality (4.4.5) in \cite{Kob} shows that for any complex surface $X$ one has
\be
\int_X{\rm Tr}(F_{\Xi_i}\wedge F_{\Xi_i})=||F_{\Xi_i}||^2_{L^2(H_0,  g_{t_i})}-||\i\Lambda_{\omega_{t_i}}F_{\Xi_t}||^2_{L^2(H_0,g_{t_i})}.\nonumber
\ee
Since $F_{\Xi_i}$ is HYM with respect to $g_{t_i}$, the right most term vanishes, and so 
\be
\mathcal{YM}(\Xi_i, g_{t_i})=\int_X{\rm Tr}(F_{\Xi_i}\wedge F_{\Xi_i}).\nonumber
\ee 
The right hand side above yields a topological invariant $[c_2( V)-\frac12c_1^2( V)]\cup X$, and is thus independent of $i$, proving the lemma.
\end{proof}

We now review the Friedman-Morgan-Witten construction of stable holomorphic bundles on elliptic fibrations with sections, and is needed for Corollary \ref{maincor}. We begin by looking at a single fiber. Let $E$ be an elliptic curve, and $0\in E$ the identity of the group law. Denote the trivial line bundle by $\cO$, and given a point $q\in E$, let $\mathcal O_E(q-0)$ be the line bundle associated to the divisor $q-0$. We also define a sequence of rank $r$ bundles (denoted $\cI_r$) inductively, with $\cI_1=\cO$ and $\cI_r$ the unique nontrivial extension of $\cI_{r-1}$ by $\cO$. Recall the following theorem of Atiyah (Theorem 5 from \cite{Atiyah}):
\begin{thm}[Atiyah \cite{Atiyah}]
Any semi-stable, degree zero bundle $V$ over $E$ is isomorphic to a direct sum of bundles of the form $ \mathcal O_E(q-0)\otimes\cI_r$, i.e.
\be
\label{form}
V\cong\bigoplus_{j=1}^\ell  \mathcal O_E(q_j-0)\otimes\cI_{r_j}.\nonumber
\ee
\end{thm}

\begin{defn}
A semi-stable bundle is called regular if in the above direct sum  $q_i\neq q_j$ for $i\neq j$. 
\end{defn}
Note that bundles of the form $ \mathcal O_E(q_j-0)\otimes\cI_{r_j}$ do not admit flat connections unless $r_j=1$. However, we can instead replace $ \mathcal O_E(q_j-0)\otimes\cI_{r_j}$ with its Seshadri filtration $ \mathcal O_E(q_j-0)^{\oplus {r_j}}$, and define all bundles with the same Seshadri filtration to be $\cS$-equivalent. We then see the $\cS$-equivalence class of an $SU(n)$ bundle is determined by $n$ points $q_1,..., q_n$ (counted with multiplicities) satisfying $q_1+\cdots+q_n=0$.

Thus we can describe the moduli space of $\cS$-equivalence classes of $SU(n)$ bundles as follows. Let $W:=H^0(E,\cO(n0))$ be the space of meromorphic functions $\phi$ that have a pole of at most order $n$ at $0$, with no other poles. By Abel's Theorem $\phi$ must have $n$ zeros satisfying $q_1+\cdots+q_n=0$. If $\phi$ has a pole of order less than $n$ at $0$, we interpret this as some of the $q_i$ are $0$. The zeros of $\phi$ are preserved under multiplication by an element of ${\mathbb C}^*$, and so the moduli space is ${\mathbb P}W\cong{\mathbb P}^{n-1}$.

Next consider a projective elliptic fibration $\pi: X\rightarrow {\mathbb P}^1$, with singular fibers of type $I_1$ or $II$ (this extra assumption gives that $X$ coincides with its Weierstrass model). Over each generic point $x$ in the base there is an elliptic curve $E_x:=\pi^{-1}(x)$ and a moduli space ${\mathbb P}^{n-1}_x$ of $SU(n)$ bundles. Friedman-Morgan-Witten demonstrate that the projective spaces glue together to form a ${\mathbb P}^{n-1}$ bundle over the base, which we denote by $\cW$. A  holomorphic $SU(n)$ bundle $ V$ over $X$ which restricts to a semi-stable bundle on each fiber determines a section $s$ of $\cW$, which in turn defines a divsor $D_V\subset X$. Specifically, each point $x$ in the base determines $n$ points in $E_x$, thus $D_V$ is an n-fold ramified cover of ${\mathbb P}^1$. More precisely, in Section 4 of  \cite{FMW2} it is demonstrated:

\begin{thm}[Friedman-Morgan-Witten \cite{FMW2}] Let $\pi: X\rightarrow {\mathbb P}^1$ be an elliptic fibration with a section $\sigma$, with singular fibers of type $I_1$ or $II$. Let $V$ be a holomorphic bundle of rank $n$ over $X$. Assume that the restriction of $V$ to a generic fiber of $\pi$ is semi-stable and regular. Then there exists a divisor 
\be
D_V\in |n\sigma(\mathbb P^1)+kl|,\nonumber
\ee 
called the spectral cover associated to $V$, where $l$ denotes the effective divisor class of the fibers of $\pi$, $k\in\mathbb Z$ satisfies $0\leq k\leq c_2(V)$. For a generic $x\in \mathbb P^1\backslash Z_\pi$, if $V|_{E_x}\cong\bigoplus_{j=1}^\ell  \mathcal O_E(q_j-0)\otimes\cI_{r_j}$, then
\be
D_V\cap E_x=\sum_{j=1}^\ell r_jq_j\in|n\sigma(x)|.\nonumber
\ee
\end{thm}

If $V$ admits a spectral cover $D_V$ which is reduced and irreducible, then $D_V$ has a finite number of ramification points. Let $Z_D$ denote the image of these ramification point under $\pi$. Then for any $x\in \mathbb P^1\backslash (Z_\pi\cup Z_D)$, we have $D_V\cap E_x=\sum_{j=1}^n q_j$ with all $q_j$ distinct, and thus
\be
 V|_{E_x}\cong  \mathcal O_E(q_1-0)\oplus \cdots \oplus \mathcal O_E(q_n-0).\nonumber
\ee
This verifies the holomorphic structure assumption on $V|_{E_x}$ in Theorem \ref{maintheorem}.
Furthermore the points $q_j$   vary holomorphically in $x$. The condition that $D_V$ be  reduced and irreducible also guarantees that the bundle $V$ is stable with respect to $\omega_t$ for small $t$. This  can be used to construct many examples of connections $\Xi_i$ that satisfy the assumptions of our main theorem.
\begin{thm}[Theorem 7.4 in \cite{FMW2}]
If the spectral cover $D_V$ constructed above is reduced and irreducible, then $V$ is stable with respect to $\pi^*[\omega_{\mathbb P^1}]+t^2[\omega_X]$ for any ample class $[\omega_X]$ on $X$,  for all $0<t^2\leq(\frac{n^3}{4}c_2(V))^{-1}.$
\end{thm}

We end this section with the construction of $\Xi_0$ from Corollary \ref{maincor}, which is  a local $HYM$ connection that determines the limit $A_0=\Xi_0|_E$ on each fiber. Although the limiting connection $A_0$ is expressed in holomorphic coordinates in \eqref{limit}, here we find it easier to work with our  coordinates $(x^1,x^2, y^1, y^2)$ from the previous section. Both viewpoints are, of course, equivalent.

Consider  $X_U:=\pi^{-1}(U)$ for some simply connected $U\subset P^1\backslash (Z_\pi\cup Z_D)$. As a first step, we consider the case where $ V$ has rank one. Since it has degree zero, the bundle $ V$ is topologically trivial along each fiber and thus topologically trivial on $X_U$. We equip $ V$ with a trivial metric $H_0$, and fix a unitary frame. For a fiber $E_x$ we have assumed the restriction $V|_{E_x}\cong \mathcal O_E(q -0)$, with   $q$ varying holomorphically in the base. We decompose $q$ as follows
\be
\label{defofq}
q(x^1,x^2)=\theta_1(x^1,x^2)-\tau\theta_2(x^1,x^2).
\ee
Recall  $E_x$ is determined by the quotient $T_xB/\Lambda$, and so the point $q$ can be lifted to a point $\ti q$ in $T_xB\cong {\mathbb C}$. Now, if $\tau(x^1, x^2)$ gives the complex structure on $E_x$, we can define the holomorphic structure on $V$ by
\be
\bar\partial_q:=\bar\partial-\frac{\pi\,\ti q}{{\rm Im}(\tau)}\,d\bar z,\nonumber
\ee
where $z$ is the complex coordiante defined in \eqref{complexcoord}. At the end of the section we will demonstrate that the connections we construct are independent of the lift from $q$ to $\ti q$, and therefore well defined. 

Given the above holomorphic structure, and using that $H_0$ is the trivial metric, the Chern connection can be computed as
\be
\Xi_0=2\pi\i\,\left(\theta_1dy^1+\theta_2dy^2\right).\nonumber
\ee
Since each $\theta_j$ only depends on the base coordinates, $\Xi_0|_{E_x}$ is flat on each fiber. Using  \eqref{defofq}, one can check that the  expression for $\Xi_0|_{E_x}=:A_0$ is equivalent to \eqref{limit}  in the statement of the main theorem. The holonomy around each period in $E_x$ is given by  $e^{2\pi\i\,\theta_1}$ and  $e^{2\pi\i\,\theta_2}$, respectively.

\begin{prop}
\label{initialconnection}
The connection $\Xi_0$ is HYM with respect to all three complex structures $I, J $, and $K$, on $X_U$. 
\end{prop}
\begin{proof}
As a first step we show
\be
\label{step1,1}
\phi^{ij}\frac{\pl}{\pl x^i}\theta_j =0
\ee
and
\be
\label{step1,2}
\frac{\pl}{\pl x^i}\theta_j=\frac{\pl}{\pl x^j}\theta_i.
\ee
To see this, because both $\tau$ and $q$ are holomorphic in the base, one can compute
\bea
0=\frac{\pl}{\pl\bar w}q&=&\frac12\left(\frac{\pl}{\pl x^1}-\bar \tau\frac{\pl}{\pl x^2}\right)(\theta_1-\tau \theta_2)\nonumber\\
&=&\frac12\left(\frac{\pl}{\pl x^1}\theta_1-\bar\tau\frac{\pl}{\pl x^2}\theta_1-\tau\frac{\pl}{\pl x^1}\theta_2+|\tau|^2\frac{\pl}{\pl x^2}\theta_2\right).\nonumber
\eea
Now, using \eqref{taudef}, the norm of $\tau$ is given by
\be
|\tau|^2=\frac{(1-\i\phi_{12})(1+\i\phi_{21})}{\phi_{22}^2}=\frac{1+\phi_{12}^2}{\phi_{22}^2}=\frac{\phi_{11}}{\phi_{22}},\nonumber
\ee
where for the last equality we used det$(\phi_{ij})=1$. Thus
\be
0=\frac1{2\phi_{22}}\left(\phi_{22}\frac{\pl}{\pl x^1}\theta_1+(\i-\phi_{12})\frac{\pl}{\pl x^2}\theta_1-(\i+\phi_{12})\frac{\pl}{\pl x^1}\theta_2+\phi_{11}\frac{\pl}{\pl x^2}\theta_2\right).\nonumber
\ee
Since both the real and imaginary parts vanish, \eqref{step1,1} and \eqref{step1,2} follow. In particular \eqref{step1,2} allows us to simplify our notation and denote $\frac{\pl}{\pl x^i}{\theta_j}$ as $\theta_{ij}$, where the indices commute.

The curvature of $\Xi_0$ is now given by
\be
F_{\Xi_0}={2\pi \i}\,\theta_{ij} dx^i\wedge dy^j.\nonumber
\ee
Right away it follows that $F_{\Xi_0}\wedge \omega_{I}=0$ for all $t$. Furthermore, \eqref{step1,1} implies $F_{\Xi_0}\wedge \omega_J=0$ and \eqref{step1,2} gives $F_{\Xi_0}\wedge \omega_{K}=0$. Thus $\Xi_0$ is a holomorphic and HYM with respect to each complex structure.

\end{proof}

We now turn to the general case. Assume $ V|_{E_x}\cong \oplus_{j=1}^n\mathcal O_E(q_j-0),$ with each $q_j$ is distinct. As before write $q_j=\theta_1^j-\tau\theta_2^j$, and construct diagonal matrices $\Theta_1$ and $\Theta_2$ with eigenvalues $\theta_1^j$ and $\theta_2^j$, respectively. Consider the connection 
\be
\label{reference connection}
{\Xi_0}=2\pi\i\left(\Theta_1dy^1+\Theta_2 dy^2\right),
\ee 
Its curvature is given by
\be
F_{\Xi_0}={2\pi \i}\,\Theta_{ij}dx^i\wedge dy^j.\nonumber
\ee
It is clear that $F_{\Xi_0}|_{E_x}=0$ for every fiber in $X_U$. Furthermore, by Proposition \ref{initialconnection} we have 
\be
F_{\Xi_0}\wedge \omega_{I}=F_{\Xi_0}\wedge \omega_J=F_{\Xi_0}\wedge \omega_K=0.\nonumber
\ee
Thus $\Xi_0$ is a local HYM connection with respect to each complex structure, although we only focus on $I$ in this paper.

Finally, we demonstrate that the lift of each point $q_i$ in $E_x$ to $\mathbb C$ is well defined. Recall that in the coordinates $(y^1,y^2)$, our lattice $\Lambda$ is the standard lattice given by $\langle 1,1\rangle$. Now, suppose we have another connection $\Psi$ satisfying
\be
 {\Xi_0}-\Psi=2\pi\i\left(M_1dy^1+M_2 dy^2\right),\nonumber
\ee
where $M_1=$diag$(\al_1,...,\al_n)$ and $M_2=$diag$(\b_1,...,\b_n)$ are both diagonal matrices of integers, which means both $\Xi_0$ and $\Psi$ define the same points on $E_x$.  If $u$ is the gauge transformation given by
\be
u={\rm diag}(e^{2\pi\i(\al_1 y^1+\b_1 y^2)},...,e^{2\pi\i(\al_n y^1+\b_n y^2)}),\nonumber
\ee
we have $ {\Xi_0}- \Psi=-duu^{-1}$. Furthermore because all the $\al_i$ and $\b_i$ are integers, $u$ descends to a smooth gauge transformation on the torus fibers, and thus the connections are gauge equivalent on $X_U$. Finally, since the points $q_i$ add up to $0$ in the group law on $E_x$, we can  find a lift to $\mathbb C$ where the points still add to $0$, and  $\Xi_0$ will be trace free.

\section{Bubbling}
\label{bubs}

We now present our bubbling argument, following the first two cases from \cite{DS}.  Consider a sequence of Hermitian-Yang-Mills connections $\Xi_i$, corresponding to $\omega_{t_i}\in[\omega_{\mathbb P^1}+t^2_{i}\omega_X]$ as $t_i\rightarrow0$. Currently our argument depends on the sequence of connections we choose, although one can hope that with further analysis the set where bubbles occur can be uniquely identified by $(V,\partial_{\Xi})$.

Choose a compact set $K\subset \mathbb P^1\backslash Z_\pi$, where as before $Z_\pi$ is the image of the singular fibers under $\pi$. In a neighborhood $U$ of any point $x\in K$, we can choose affine coordinates $(x^1,x^2)$ where $x$ is at the origin, and coordinates $(x^1,x^2,y^1,y^2)$ on $X_U:=\pi^{-1}(U)$.  
The curvature of $\Xi_i$ can   be decomposed as $$F_{\Xi_i} = F_{B_i} + F_{A_i} + \kappa_i,$$ where $F_{B_i}$ and $F_{A_i}$ denote the base and fiber directions of the curvature, and $\kappa_i$ denotes the mixed terms. For  each $x\in K$ we define the quantity
\be
m_i(x):=||F_{B_i}||_{L^\infty(E_x,H_0,g_{X})} + \frac{1}{t_i^2}||F_{A_i}||_{L^\infty(E_x,H_0,g_{X})}+||\kappa_i||^2_{L^\infty(E_x, H_0, g_{X})},\nonumber
\ee
where $g_{X}$ is the metric associated to the fixed K\"ahler form $\omega_{X}$.

\begin{prop}
\label{bubbling}
There exists a finite number of points $\{p_1,...,p_\ell\}\subset K$ such that for any compact set $K'\subset K\setminus\{p_1,\cdots,p_\ell\}$, $$\lim_{i\rightarrow \infty} t_i^2||m_i||_{L^\infty(K')} = 0.$$ In particular, for any $x\in K\setminus\{p_1,\cdots,p_\ell\}$, $$\lim_{i\rightarrow \infty}||F_{A_i}||_{L^\infty(E_x)} = 0.$$
\end{prop}
\begin{proof}
Let $x_i$ be a sequence of points in $K$ for which $t_i^2 m_i(x_i)$ does not approach zero. We will show  a finite amount of energy must bubble off along this sequence. Thus, by the total energy bound (Lemma \ref{YMcontrol}), there can only be a finite number of points in $K$ where bubbling occurs. 

The proof closely follows the arguments of \cite{Chen, DS, N1, N2} and is divided into two cases. The first case occurs when $t_i^2m_i(x_i)$ is unbounded, and the second when $t_i^2m_i(x_i)$ stays bounded above, yet is also bounded away from zero. 
Unless mentioned otherwise, all bundle norms in this section are with respect to  $H_0$, so  we suppress $H_0$ from our notation for simplicity.

\medskip
\noindent{\bf Case 1.}  $t_i^2 m_i(x_i)$ is unbounded.

\medskip

Given our sequence of points $x_i$, there exists corresponding points $a_i$ in $E_{x_i}$ where the supremum is obtained,  and without loss of generality we can assume $(x_i,a_i)\rightarrow (x_0,a_0)$. Let $D_{r}(x_i)$ denote a disc of radius $r$ in the metric $g_{\mathbb{P}^1}$ (corresponding to $\omega_{\mathbb P^1}$) in the base. We will show there exists a universal constant $\epsilon_0>0$, so that the inequality 
\be
\label{bub1}
\liminf_{i\rightarrow \infty}\int_{\pi^{-1}(D_r(x_i))}|F_{\Xi_i}|^2_{g_i}dV_{g_i}>\epsilon_0
\ee
holds for any small $r>0.$  Since the total energy is finite, a standard covering argument then shows that there can only be finitely many bubbles of this type.

Suppose \eqref{bub1} does not hold. Then there exists an $r_0$ so that for sufficiently large $i$,
\be
\label{contradictioncase1}
\int_{\pi^{-1}(D_{r_0}(x_i))}|F_{\Xi_i}|^2_{g_i}dV_{g_i}\leq \epsilon_0.\nonumber
\ee
It follows from \eqref{equiv} that there is a universal constant $c>0$, so that the $g_i$-geodesic ball $B_i :=B_{cr_0 {t_i}}(x_i,a_i)$ is contained in $\pi^{-1}(D_{r_0}(x_i))$. In particular we have the following bound 
\be\label{contradictioncase1.1}
\int_{B_i}|F_{\Xi_i}|^2_{g_i}dV_{g_i}\leq \epsilon_0.\nonumber
\ee

We now rescale our coordinates and metrics. Consider the coordinate change 
$$\lambda_i( x,y) = \Big({t_i}~(  x+x_i),y+a_i\Big),$$
 and let $\tilde \omega_i= t_i^{-2}\lambda_i^*\omega_i$. To compare this to the scaling  $L_{t_i}$,  defined in \eqref{fiberscale} in Section \ref{semiflat}, note that pulling the metric back by $L_{t_i}$ dilates the shrinking fibers, while $\tilde \omega_i$ is a combination of shrinking the base coordinates and then dilating the metric. Thus both scalings have the same effect, although with different coordinates.
 
By \eqref{equiv},  $\tilde \omega_i$ is uniformly equivalent to the Euclidean metric in the scaled coordinates, which we denote by $\tilde g_0$. Moreover the ball $B_i$ pulls back to a $\tilde g_i$-geodesic ball $\tilde B_i$, which contains a Euclidean ball $\tilde B.$ The ball $\tilde B$ can be chosen to have uniform size independent of $i$.

Now, if $\tilde \Xi_i$ is the pull-back connection to these new coordinates, then the HYM equation is again satisfied
\be
\label{HYMscaled}
F_{\tilde \Xi_i}\wedge \tilde\omega_i = 0\qquad{\rm and}\qquad (F_{\tilde\Xi})^{0,2}=0.
\ee
Change of variables, and the scale invariance of the Yang-Mills energy in dimension four, implies
\be
\label{L^2equiv}
 ||F_{\ti \Xi_i}||^2_{L^2(\ti B,\ti g_0)}\leq  ||F_{\ti \Xi_i}||^2_{L^2(\ti B_i,\ti g_i)}=  ||F_{ \Xi_i}||^2_{L^2(B_i, g_i)} \leq \epsilon_0.\nonumber
\ee
Since $\ti g_i$ is uniformly equivalent to  $\ti g_0$ on $\tilde B$ for large $i$, and $\ti\Xi_i$ satisfies \eqref{HYMscaled}, we can apply the standard $\epsilon$-regularity argument for Yang-Mills connections on a fixed ball $\ti B$  \cite[Theorem 4.8]{U1}. Thus, for $\epsilon_0$ small enough (depending only on the real dimension $4$ of X), the above $L^2$ control implies
\be
|F_{\ti\Xi_i}|^2_{\ti g_{i}}(0)\leq C,\nonumber
\ee
for some constant $C$ independent of $i$. Equivalence of metrics implies  control of $F_{\ti\Xi_i}$ with respect to    $\tilde g_0$, which in components gives the following bound
\be
|F_{\ti B_i}|_{\ti g_0}(0) +  |F_{\ti A_i}|_{\ti g_0}(0) + |\ti \kappa_i|^2_{\ti g_0}(0)\leq C.\nonumber
\ee
Scaling back, we see
\be
t_i^2|F_{B_i}|_{g_X}(0) + |F_{A_i}|_{g_X}(0) +t_i^2|\kappa_i|^2_{g_X}(0)\leq C.\nonumber
\ee
Hence we achieve control of $t_i^2 m_i(0)$, which we have assumed diverges, a contradiction. Let $W_1$ denote the set of points in $K$ at which bubbles of this type appear.

\medskip
\noindent{\bf Case 2.} $t_i^2m_i(x_i)$ is bounded above and away from zero.

\medskip

In this case an instanton on ${\mathbb C}\times E$ bubbles off. We follow the outline of \cite{N1,N2}, and use an energy quantization result of Wehrheim. 

Suppose $x_i\rightarrow x_0 \in K\backslash W_1$, and  let $D_{2\rho}(x_i)$ denote a disc of radius $2\rho$ in the metric $g_{\mathbb{P}^1}$ in the base. Suppose there exists constants $\delta$ and $\Lambda$ so that 
$$\delta<t_i^2 m_i(x_i)\leq\sup_{D_{2\rho}(x_i)}t_i^2 m_i<\Lambda.$$
The rightmost inequality holds since for large enough $i$ we can assume $D_{2\rho}(x_i)\subset K\backslash W_1$. By making $\rho$ smaller if necessary, we can furthermore assume that $\pi^{-1}(D_{2\rho}(x_i))$ is topologically a product between a ball in $\mathbb C$ and an elliptic curve $E$, although the complex structure may vary.

We preform the same scaling as in Case 1, and define $\tilde \omega_i= t_i^{-2}\lambda_i^*\omega_i$. Again this involves  rescaling  the metric and applying a   dilation.  The disk $D_{2\rho}(x_i)$ pulls back to $\ti D_{{2\rho}/{{t_i}}}(0)$, the geodesic disk with respect to the Euclidean metric $\ti g_0$ in the scaled coordinates. Starting from Proposition \ref{HTbound}, the arguments used in the proof of \cite[Theorem 1.1]{TZ0} (cf. pages 2936-2937)  give that $\tilde \omega_i$ converge sub-sequentially and smoothly to a limiting flat product metric $\omega_\infty$ on $\mathbb{C}\times E$.

Our sequence of scaled connections $\ti \Xi_i$ is defined on $\pi^{-1}(\ti D_{{2\rho}/{ {t_i}}}(0))$, and for any point $p\in\ti D_{{2\rho}/{ {t_i}}}(0)$, we have
 \begin{align*}
 |F_{\ti B_i}|_{\ti g_0}(p) + |F_{\ti A_i}|_{\ti g_0}(p)+ |\ti \kappa_i|^2_{ \ti g_0}(p)&=t_i^2 |F_{B_i}|_{ g_0}(p)+|F_{A_i}|_{g_0}(p)+t_i^2|\kappa_i|^2_{  g_0}(p)\\
 &<\sup_{D_{2\rho}(x_i)}t_i ^2m_i < \Lambda.
 \end{align*}
This implies $| F_{\ti\Xi_i}|_{\ti g_i}$ is uniformly bounded. By strong Uhlenbeck compactness \cite[Corollary 1.4 and Theorem 1.5]{U2}, this bound implies  there exists a subsequence of connections which converges smoothly, modulo unitary gauge transformations, to a limiting connection $\ti \Xi_\infty$ on the trivial $SU(n)$ bundle over $\mathbb{C}\times E$. The connection $\tilde\Xi_\infty$ will be ASD with respect to the limiting product metric $\omega_\infty$. Furthermore, by assumption, we have
$$|F_{\ti B_i}|_{\ti g_0}(0)+|F_{\ti A_i}|_{\ti g_0}(0)+ |\ti \kappa_i|^2_{ \ti g_0}(0)>\delta,$$
 and it follows that the limiting connection is not flat. An energy quantization result of Wehrheim  \cite[Remark 1.2]{W1} implies there exists a universal constant $\epsilon_0>0$ so that 
 $$\mathcal{YM}(\ti \Xi_\infty,\ti g_0)\geq 2\epsilon_0.$$
This implies that there exists an $R>0$, so that for $i$ sufficiently large, 
$$\epsilon_0<\int_{\ti D_{\frac R 2}(0)\times E}|F_{\ti\Xi_i}|^2_{\ti g_i}dV_{\ti g_i}=\int_{\pi^{-1}(D_{R{t_i}}(x_i))}|F_{\Xi_i}|^2_{g_i}dV_i.$$
Thus there can be only finitely many bubbles of this type, and denote the set of all such bubbles by $W_2$. This concludes the proof of Proposition \ref{bubbling}. 
\end{proof}

We conclude this section by noting that Proposition \ref{bubbling}, in conjunction with our convergence argument in Section \ref{convo}, implies that on a generic fiber, the connections $A_i$ will converge to a limiting flat connection $A_0$. The following section  is devoted to identifying this limiting flat connection explicitly.

\section{Gauge fixing over an elliptic curve}

\label{gaugefixingsection}

In this section we work on a fixed fiber of $\pi$, denoted $E$ for simplicity, satisfying  $V|_E\cong\bigoplus_{j=1}^n  \mathcal O_E(q_j-0)$ with each $q_j$ distinct and  $q_1+\cdots+q_n=0$. By the results of Section \ref{bundles}, if $V$ defines a spectral cover $D_V$ which is reduced and irreducible, this happens generically.

Equip $E$ with the fixed K\"ahler form $\omega_0=dy^1\wedge dy^2$ and let $g_0$ denote the corresponding metric. Recall that $E$ carries the complex coordinate $z=\tau y^1+y^2$. Denote the restriction  $ V|_E$ by $ V_0$. Since $V_0$ is of the form  $\bigoplus_{j=1}^n  \mathcal O_E(q_j-0)$,  it can be naturally identified with $E\times \mathbb C^n$, equipped with the complex structure
\be
\bar\partial_{A_0}:=\bar\partial-\frac{\pi\,Q}{{\rm Im}(\tau)}\,d\bar z,\nonumber
\ee 
where $Q$ is a diagonal matrix with entries $\ti q_i$ (recall $\ti q_i$  are the lifts of the points $q_i$ to $\mathbb C$). Let $H_0$ be the trivial metric on $\mathbb C^n$, and let $A_0$ be the Chern connection associated to $\bar\partial_{A_0}$ and $H_0$. Using \eqref{defofq}, in addition to $dz|_E =\tau dy^1+dy^2$, one can explicitly check that that $A_0=\Xi_0|_E$, where   $\Xi_0$ is given by \eqref{reference connection}.

Now,  given our sequence of connections $\Xi_i$ on $V$, we have the sequence of restricted connections $A_i:=\Xi_i|_E$ on $ V_0$. Again, because our sequence of connections arises from the Donaldson-Uhlenbeck-Yau Theorem, we know that  $A_i$ lies in the complexified gauge orbit of $A_0$, and thus $A_i$ and $A_0$ define isomorphic complex structures.  As a result, after   transforming $A_0$ by a unitary gauge transformation if necessary, we can write $A_i=e^{s_i}( A_0)$ for a trace free Hermitian endomorphism $s_i$. The main result of this section is:
\begin{thm}
\label{gaugef}
Let  $e^s(A_0)$ be a connection on $ V_0$ given by the action of a trace free Hermitian endomorphism $s.$ There exists constants $\epsilon_0>0$, and $C_0>0$, depending only on $g_0$,  $A_0,$ and $H_0$, so that the following holds. If the curvature of $e^s(A_0)$ satisfies
\be
\qquad ||F_{e^s(A_0)}||^2_{C^0(g_0, H_0)}\leq\epsilon_0,\nonumber
\ee 
then there exists another trace free Hermitian endomorphism $s'$ which satisfied both 
\be
e^s( A_0)=e^{s'} ( A_0)\qquad{\rm and}\qquad ||s'||_{C^0(g_0, H_0) }\leq C_0.\nonumber
\ee
\end{thm}

Because $ V_0$ is poly-stable, and not stable, the above theorem is the best $C^0$ control that one can expect. For example, since $A_0$ is flat, if $e^s$ is a diagonal matrix of constants $c_1,..., c_n$, then $e^s ( A_0)$ will still be flat. However, one eigenvalue $c_i$ can be arbitrarily large while still preserving the condition det$(e^s)=1$ (recall that $e^s\in SL( V_0)$). Thus one can never expect $C^0$ control for $s$. The main idea of the above theorem is that, by a suitable choice of normalization, one can construct a related complex gauge transformation that yields the same connection, yet with the desired $C^0$ control.

We first demonstrate several preliminary results. For the remainder of the section, unless specified, all norms are taken with respect to the metrics $g_0$ and $H_0$, and we remove this from our notation for simplicity.

\begin{prop}
\label{poincare} Let $s$ be a trace free Hermitian endomorphism such that the diagonal entries of $s$ have zero average when integrated over $E$.  Then there exists a constant $C_p$, independent of $s$, so that
\be
\label{awesome1}
||s||_{L^2(E)}\leq C_p||\bar\partial_{A_0}s||_{L^2(E)}.
\ee
\end{prop}
\begin{proof}
Note that diagonal entries of $s$ can always be defined  as the entries that preserve each subbundle  $\mathcal O_E(q_i-0)\subset V_0$.  Now, assume that the inequality does not hold. Then there exists a sequence of endomorphisms $s_k$ satisfying the assumptions of the proposition, along with the inequality 
\be
\int_E|\bar\partial_{A_0} s_k|^2\leq \frac1k\int_E |s_k|^2.\nonumber
\ee 
Let $\ti s_k:=s_k/||s_k||_{L^2(E)}$. Then
\be
\int_E|\bar\partial_{A_0} \ti s_k|^2\leq \frac1k.\nonumber
\ee
Because $s$ is Hermitian with respect to $H_0$, we have $|\bar\partial_{A_0} \ti s_k|=|\partial_{A_0} \ti s_k|$. This allows us to conclude that the sequence $\ti s_k$ converges weakly in $L^2_1$ (and strongly in $L^2$) to an endomorphism $s_\infty$ satisfying 
\be
||s_\infty||_{L^2(E)}=1\qquad{\rm and}\qquad ||\bar\partial_{A_0} s_\infty||_{L^2(E)}=0.\nonumber
\ee

Now, because $Q$ is diagonal, if $s_\infty$ has entries $(a_{ij})$, then the diagonal entries of $\bar\partial_{A_0} s_\infty$ are of the form $\bar\partial \,a_{ii}$. Since $||\bar\partial_{A_0} s_\infty||_{L^2}=0$, we see the diagonal entries of $s_\infty$ are constant. Furthermore, by assumption the diagonal entries of $\ti s_k$ have zero average, and by strong convergence in $L^2$ we conclude the diagonal entries of $s_\infty$ must also have zero average. Thus these entries vanish entirely. Now, because the points $q_i$ are distinct, the automorphism group of $ V_0$ is precisely $n$ dimensional \cite[Lemma 1.13]{FMW2}. Thus if the diagonal entries of $s_\infty$ vanish, $s_\infty$ must vanish entirely. In other words, if $s_\infty$ had any non-vanishing off diagonal entries, they would define a holomorphic map between line bundles $ \mathcal O_E(q_i-0)$ and $ \mathcal O_E(q_j-0)$ for distinct points $q_i$ and $q_j$, which is impossible. So $s_\infty=0$ yet $||s_\infty||_{L^2(E)}=1$, a contradiction. 
\end{proof}

Next consider  the following function spaces, equipped with the $L^2$ norm. 
\begin{defn}
Let ${\rm Herm}_0( V_0)$ be the space of trace free Hermitian endomorphisms of $ V_0$, and furthermore let ${\rm Herm}^\perp_0( V_0)$ denote the subspace consisting of those endomorphisms whose diagonal entries   have zero average on $E$.  
\end{defn}

Consider $\Upsilon(\cdot)\in {\rm End}({\mathfrak gl}( V_0))$, defined by
\be
\Upsilon(s)=\frac{e^{{\rm ad}_s}-1}{{\rm ad}_s}.\nonumber
\ee
From the definition of the complexified gauge action \eqref{complexifiedgauge}, we have
\be
\label{action2}
\partial_{e^{s}( A_0)}=\partial_{A_0}+e^{-s}(\partial_{A_0} e^{s})\qquad{\rm and}\qquad\bar\partial_{e^{s} ( A_0)}=\bar\partial_{A_0}-(\bar\partial_{A_0} e^{s})e^{-s}.
\ee
This allows one to compute
\be
e^{s} (A_0)=A_0+\Upsilon({-s})\partial_{A_0} s-\Upsilon({s})\bar\partial_{A_0} s\label{definitionofN}
\ee
(for instance, see \cite[Appendix A]{JW}).

Define the map ${\mathcal N(s)}:=e^{s} (A_0)$, which maps ${\rm Herm}_0( V_0)$ into the affine space of connections centered at $A_0$, equipped with the $L^2$ norm. Let ${\mathcal A}$ denote the image of the map $\mathcal N$. Using \eqref{definitionofN}, we see the derivative of $\mathcal N$ at $0$ is given by
\be
L(s):={\mathcal N}'(0)(s)=\partial_{A_0} s-\bar\partial_{A_0} s.\nonumber
\ee
The tangent space to ${\rm Herm}_0( V_0)$ is again ${\rm Herm}_0( V_0)$. Note that for any $s\in {\rm Herm}_0( V_0)$, if $\hat s$ is a diagonal matrix of constants given by the averaging the diagonals of $s$ over $E$, then $L(s)=L(s-\hat s)$, and so both ${\rm Herm}_0( V_0)$ and ${\rm Herm}^\perp_0( V_0)$ have the same image under $L$. Proposition \ref{poincare} shows that $L$ has trivial Kernel on ${\rm Herm}^\perp_0( V_0)$. Thus, not only can we conclude that the restriction of $L$ to ${\rm Herm}^\perp_0( V_0)$ is invertible, but  \eqref{awesome1} shows in addition that $L$ has bounded inverse. The contraction mapping principle implies there exists a small neighborhood ${\mathcal U}\subset {\mathcal A}$ of the connection $A_0$, and a set $\mathcal V\subset {\rm Herm}^\perp_0( V_0)$ in the tangent space to ${\rm Herm}_0( V_0)$ at $0$, so that $ {\mathcal V}\longrightarrow{\mathcal U}$ is a diffeomorphism onto its image. 

Summing up, we have proved the following:

\begin{lem}\label{implicit function}
There exist constants $\delta_0$ and $\Lambda_0$, which depend only on $A_0$, $H_0$ and $g_0$, so that the following holds. If $A\in \mathcal{A}$ satisfies $||A - A_{0}||_{L^2(E)}< \delta_0$, there exists $s\in {\rm Herm}_0^\perp$, such that $$A = e^{s} (A_{0}),$$ and $$||s||_{L^2(E) }\leq \Lambda_0||e^{s} (A_0)-A_0||_{L^2(E) }.$$
\end{lem} 

We now turn to one final lemma. Consider the same constant $\delta_0>0$ from above, and let $C_0>0$ be a fixed constant, to be determined in the proof of Theorem \ref{gaugef}.
\begin{lem}
\label{connection estimate}
Let $s$ be a trace free Hermitian endomorphism, and $A$ a flat connection on $ V_0$. Given constants $\delta_0>0$ and $C_0>0$, there exists a constant $\epsilon_0$, depending only on $H_0$, $g_0$, $\delta_0$, and $C_0$, so that if $||s||_{C^0 }\leq C_0$ and $||F_{e^{s} (A)}||_{C^0(E)}<\epsilon_0$, then
\be
||e^{s} (A)-A||_{L^2(E)}<\frac{\delta_0}2.\nonumber
\ee
\end{lem}
\begin{proof}
To begin, we see how the curvature of $A$ is related to the curvature of $e^{s} (A)$. Using \eqref{action2} one can compute
\be
e^{-s}F_{e^{s} (A)} e^{s}-F_{A}=\bar\partial_A(e^{-2s}\partial_A e^{2s})=e^{-2s}\left(\bar\partial_A\partial_A e^{2s}-\bar\partial_Ae^{2s} e^{-2s}\partial_A e^{2s}\right),\nonumber
\ee
which implies
\be
\label{curvatures}
-\Delta_{g_0}{\rm Tr}(e^{2s})+|e^{-s}\partial_A e^{2s}|^2_{g_0}={\rm Tr}(e^{2s}i\Lambda_{\omega_0} F_{e^{s} (A)}).
\ee
Integrating over $E$ yields
\be
\int_E |e^{-s}\partial_A e^{2s}|^2<e^{C_0}\epsilon_0.\nonumber
\ee
The $C_0$ bound for $s$, along with det$(e^s)=1$, demonstrates the eigenvalues of $e^s$ are bounded above and below. Thus the left hand side above controls the $L^2$ norm of the difference $e^{s}(A)-A$, and so
\be
||e^{s}(A)-A||_{L^2(E)}<C\epsilon_0.\nonumber
\ee
Here $C$ only depends on $C_0$, $g_0$, and $H_0$. Choose $\epsilon_0$ small  so $C\epsilon_0<\delta_0/2$.
\end{proof}
As we turn to the proof of Theorem \ref{gaugef}, recall that the constants $\delta_0$ and $\Lambda_0$ depend only on $A_0$, $H_0$ and $g_0$. Thus, by the above lemma, if we can show $C_0$ depends only on these quantities, $\epsilon_0$ will depend only on these quantities. 
\begin{proof}[Proof of Theorem \ref{gaugef}]

Our first task is to specify $C_0$. Fix an endomorphism $s\in{\rm Herm}_0( V_0)$ satisfying
\be
\qquad ||F_{e^s(A_0)}||^2_{C^0}\leq\epsilon_0.\nonumber
\ee 
Using this curvature bound, along with the inequality 
\be
-\Delta_{g_0}|s|^2\leq |s| |F_{e^{s} (A_0)}|\nonumber
\ee
(see for instance Proposition A.6 in \cite{JW}), we can apply Moser iteration to conclude
\be
\label{Moser}
\max\{|s|^2,1\}\leq C_1||s||_{L^2 }(1+\epsilon_0).
\ee
Here $C_1$ only depends on $g_0$ and $H_0$. We now set $C_0:= 2C_1\Lambda_0$. This shows $C_0$, and subsequently $\epsilon_0$, depends only on the initial setup.

The main idea of the proof is as follows. We construct a path of Hermitian endomorphisms, so that the curvature of the induced connections along this  path is bounded by $\epsilon_0$. We show the endpoint  of our path satisfies the conclusion of the theorem, and then apply a  method of continuity argument to conclude our desired result for $s$. Naively, one may first try to connect $e^{s}$ to $Id_{V_0}$ by the path $e^{ts}$ for $t\in[0,1]$. However, for arbitrary initial $s$ it is not clear that curvature stays bounded by $\epsilon_0$ along this path, which is an important for the argument. Instead, we follow the Yang-Mills flow.

Following Donaldson \cite[Section 1.1]{Don1}, we consider a path of complex gauge transformations $g(t)$, satisfying 
\be
\label{KNflow}
\dot g(t) g(t)^{-1}=-i\star F_{g(t) (A_0)}\qquad \qquad g(0)=e^{s}.\nonumber
\ee
On a Riemann surface, the above flow is referred to as the Kempf-Ness flow, and given a solution, the corresponding connections $A(t):=g(t) (A_0)$ solve the Yang-Mills heat flow:
\be
\dot A(t)=-d^*_{A(t)}F_{A(t)}\qquad\qquad A(0)=e^{s} (A_0).\nonumber
\ee
By a Theorem of R$\mathring{\rm a}$de \cite[Theorem 2]{R}, there exists a limiting Yang-Mills connection $A_\infty$ for which
\be
||A(t)-A_\infty||_{L^2_1}\leq c\,t^{-\beta}.\nonumber
\ee
Since $ V_0$ is polystable, the limit connection $A_\infty$ is also flat, and in the unitary gauge orbit of $A_0$, and so there exists a unitary gauge transformation $u_\infty$ for which $A_\infty=u_\infty^* A_0$.

Consider the trivial flow $u_\infty(t)=u_\infty$, which again satisfies the Kempf-Ness flow equation
\be
\dot u_\infty(t)u_\infty(t)^{-1}=-i\star F_{u_\infty(t) ( A_0)}=-i\star F_{A_\infty}=0.\nonumber
\ee
Define $\eta(t)\in {\rm Herm}_0( V_0)$, and a path of unitary gauge transformations $u(t)$, by the equation
\be
\label{relation}
u_\infty u(t)e^{\eta(t)}=g(t).
\ee
Thus $u(t)e^{\eta(t)}$ relates our two solutions of the Kempf-Ness flow. By Proposition 4.13 in \cite{Tra}, both $u(t)$ and $\eta(t)$ are bounded in $L^2_2$. In fact, this $L^2_2$ bound is proven following the general argument of \cite{GRS}. The authors demonstrate that   if $M$ is a complete, connected, simply connected Riemannian
manifold of nonpositive sectional curvature, which admits a function $\Phi:M\rightarrow \mathbb R$ which is convex along geodesics, then given two negative gradient flow lines of $\Phi$, the geodesic distance between these two flow lines stays bounded. In our case, the role of $M$ is taken by the space of Hermitian metrics, and the function $\Phi$ is Donaldson's functional (see \cite{Don1,GRS, Tra}, for a precise definition of $\Phi$).

 The bound on $\eta(t)$, along with convergence of $F_{A(t)}$ to zero in  $L^2$, allows us to conclude by Lemma \ref{connection estimate} that there exists a $T$ sufficiently large, so that
\be
||(u_\infty u(T) e^{\eta(T)}u(T)^{-1}u_\infty^{-1})\left( (u_\infty u(T))^*A_0\right)-(u_\infty u(T))^*A_0 ||_{L^2}\leq\frac{\delta_0}2.\nonumber
\ee
For simplicity we denote the fixed unitary gauge transformation $u_\infty u(T)$ by $u$, and define the path of Hermitian endomorphism $e^{\kappa(t)}$ by $u e^{\eta(t)}u^{-1}$. Then the above estimate can be written
\be
\label{largeT}
||e^{\kappa(T)} (u^*A_0)-u^*A_0 ||_{L^2}\leq\frac{\delta_0}2.
\ee
It is along the path $e^{\kappa(t)}$ that we can now apply our method of continuity argument.

Let $\ti A(t)$ be the path of connections given by $e^{\kappa(t)} (u^*A_0)$. Since $\ti A(t)=(u_\infty u(t)u^{-1})^*A(t)$, where $A(t)$ solves the Yang-Mills flow and $u(t)$ is given by \eqref{relation}, we conclude that $||F_{\ti A(t)}||_{C^0(E)}\leq\epsilon_0$ for all $t\in[0,T]$. This follows because the curvature is decreasing along the Yang-Mills flow, and the action of a unitary gauge transformation will not affect this norm. Also, the path $\ti A(t)$ is smooth for $t\in[0,T]$.

To set up the method of continuity, consider the set $I\subseteq [0,T]$ consisting of times $t$ for which there exists a  trace free Hermitian endomorphism $\kappa'(t)$ which satisfies both 
\be
\label{gauge1'}
e^{\kappa'(t)} (u^*A_0)=e^{\kappa(t)} (u^*A_0)
\ee
and 
\be
\label{gauge2'}
||\kappa'(t)||_{C^0(g_0,H_0)}\leq C_0.
\ee
We prove $I=[0,T]$. First, we demonstrate $T\in I$ to conclude $I$ is non-empty. By the estimate \eqref{largeT}, we can apply Lemma \ref{implicit function} to $\ti A(T)$, and conclude there exists a trace free Hermitian endomorphism $\kappa '(T)$ satisfying both \eqref{gauge1'} and an $L^2$ bound. By our Moser iteration bound this $L^2$ control can be improved to $C_0,$ and so $\kappa'(T)$ satisfies  \eqref{gauge2'} as well. Thus $T\in I$. We now need that $I$ is both open and closed with respect to the topology induced from the $C^2$ topology on Herm$( V_0)$. For the rest of the proof we use the notation $A:= u^*A_0$.

Our next step is to show $I$ is open. Let $t_0\in I$, and consider the corresponding endomorphism $e^{\kappa_0}$. Construct a small neighborhood of $e^{\kappa_0}$ with radius $\rho>0$, where $\rho$ is chosen so 
$$||e^{\kappa}-e^{\kappa_0}||_{C^2(E) }<\rho$$
implies 
$$||e^{\kappa} ( A) - e^{\kappa_0} ( A)||_{L^{2}(E) } < \delta_0/2.$$
Now, because $e^{\kappa_0}\in I$, there exists an endomorphism $e^{\kappa'_0}$ satisfying both $e^{\kappa_0} ( A)= e^{\kappa_0'} (A)$ and $||\kappa_0'||_{C^0 } \leq C_0.$ Given our choice of $\epsilon_0$, by Lemma \ref{connection estimate} we have
 $$||e^{\kappa_0}(A)- A||_{L^2 } =||e^{\kappa'_0}(A)- A||_{L^2 } < \delta_0/2.$$ 
By the triangle inequality  $||e^{\kappa} ( A_0) - A_{0}||_{L^2 } < \delta_0$, and  thus Lemma \ref{implicit function} implies there exists an endomorphism  $\kappa'\in {\rm Herm}_0^\perp$ such that $e^{\kappa} ( A)=e^{\kappa'} ( A)$ and 
 $$||\kappa'||_{L^2 }\leq \Lambda_0||e^{\kappa'} ( A)-A||_{L^2 }<\Lambda_0$$ 
(we assumed $\delta_0<1$). By our Moser iteration bound \eqref{Moser}, we conclude
 \be
 ||\kappa'||_{C^0 }<2\Lambda_0C_1 =C_0,\nonumber
 \ee
 which completes the proof of openness.

Finally we prove $I$ is closed. Let $t_i$ be a sequence of times in $I$ converging to $t$, and let $\kappa_i$ be the corresponding  sequence of endomorphisms converging to $ \kappa$ in the $C^2$ topology. For each $i$, there exists $ \kappa_i'$ which are uniformly bounded in $C^0$ and satisfy  $e^{\kappa_i} ( A)= e^{\kappa_i'} (A)$. The complexified gauge action gives
\be
e^{ \kappa _i}\circ\bar\partial_{ A}\circ e^{- \kappa_i}=e^{ \kappa_i'}\circ\bar\partial_{ A}\circ e^{- \kappa_i'},\nonumber
\ee
from which we conclude $\bar\partial_{A}(e^{- \kappa_i}e^{ \kappa_i'})=0$. Since $A=u^*A_0$, this implies $\bar\partial_{A_0}(u^{-1}e^{- \kappa_i}e^{ \kappa_i'}u)=0$. Since $A_0$ has only diagonal entries, we see the diagonal entries of  $u^{-1}e^{- \kappa_i}e^{ \kappa_i'}u$ must be constant. As before it then follows that the off diagonal terms must vanish, otherwise one would have a holomorphic map between line bundles $ \mathcal O_E(q_i-0)$ and $ \mathcal O_E(q_j-0)$ for distinct points $q_i$ and $q_j$. Taking the complex conjugate and using the fact that $\kappa_i$ and $\kappa_i'$ are Hermitian, we see that $ h_i:=u^{-1}e^{ \kappa_i'}e^{- \kappa_i}u$ will be a diagonal matrix of constants. 

Now, $C^2$ convergence of $ \kappa_i$ together with the $C^0$ control of $ \kappa_i'$ gives that the matrices $ h_i$ are uniformly bounded above and below. Since each $ h_i$ is a diagonal matrix of constants, after passing to a subsequence the $ h_i$ converge in $C^0$ to a limit $ h$. This allows us to define an endomorphism $e^{ \kappa'}:= uh u^{-1}e^{ \kappa}$, and by convergence of $ h_i$ and $ \kappa_i$ we have 
\be
e^{ \kappa_i'}= uh_iu^{-1} e^{ \kappa_i}\rightarrow e^{ \kappa'}\nonumber
\ee
 in $C^0$. Since $e^{\kappa_i'}$ is Hermitian, so is $e^{ \kappa'}$, and thus $ \kappa'$ is Hermitian and satisfies \eqref{gauge2'}. Furthermore, since for each $i$ we have $\bar\partial_A(u h_iu^{-1})=0$, we in fact have $ u h_iu^{-1} e^{ \kappa _i}\rightarrow e^{ \kappa'}$ in $C^1$, and as a result we conclude $e^{ \kappa'} ( A)= e^{ \kappa} (A)$. Thus \eqref{gauge1'} is also satisfied and $t\in I$. 

Thus $I$ is open, closed, and nonempty, and as a result $e^{\kappa(0)}\in I$. In particular, there exists a $\kappa'(0)$ satisfying  both $e^{\kappa(0)} (A)= e^{\kappa(0)'} (A)$ and $||\kappa(0)'||_{C^0 } \leq C_0.$ Now, define $s'$ by 
\be
e^{s'}=u^{-1}e^{\kappa(0)'}u.\nonumber
\ee
We see $s'$ satisfies the desired $C^0$ bound. Furthermore,
\bea
e^{s'} (A_0)=u^{-1}{}^*e^{\kappa(0)'} (u^*A_0)= u^{-1}{}^*e^{\kappa(0)} (u^*A_0)&=& e^{\eta(0)} ( A_0).\nonumber
\eea
Yet we started the flow \eqref{KNflow} at $g(0)=e^{s}$,  so $\eta(0)=s$. This completes the proof of the theorem.
\end{proof}

\section{Convergence}
\label{convo}
In this section we complete the proof of Theorem \ref{maintheorem}. As before, let $Z_\pi$ denote image of the singular fibers under $\pi$,  and $W_1$ and $W_2$ the bubbling sets for our sequence of connections. We fix a point $x\in \mathbb P^1\backslash (Z_\pi \cup W_1\cup W_2)$, and  denote the fiber over $x$ by  $E:=\pi^{-1}(x)$. We use the notation $A_i:=\Xi_i|_{E}$  and $A_0:=\Xi_0|_{E}$. As above equip $E$ with the fixed flat metric $\omega_0=dy^1\wedge dy^2$. Unless otherwise specified, in this section all norms are taken with respect to the metrics $g_0$ and $H_0$.

Recall our bubbling sequence at $x$ is defined by
\be
m_i(x):=||F_{B_i}||_{C^0(E,g_X)}+\frac{1}{t_i^2}||F_{A_i}||_{C^0(E,g_X)}+||\kappa_i||^2_{C^0(E, g_X)}.\nonumber
\ee
Since $x\notin W_1\cup W_2$, we have $t_i^2 m_i\rightarrow 0$ as $i\rightarrow\infty$, so $||F_{A_i}||_{C^0(E)}\rightarrow 0$. As mentioned at the end of Section \ref{bubs} (and as we shall see below), this is enough to prove that $A_i$ converges, along a subsequence and modulo gauge transformations, to a limiting flat connection. Our main result is identifying this limit.

Assume $V|_E= \mathcal O_E(q_1-0)\oplus\cdots\oplus \mathcal O_E(q_n-0)$ for $q_1,...,q_n$ distinct. Writing $A_i=e^{s_i}(A_0)$ for a sequence of Hermitian endomorphism $s_i$, for $i$ large enough we can apply Theorem \ref{gaugef}  to conclude there exists gauge transformations $s_i'$, which are uniformly bounded in $C^0$, and  satisfy $A_i=e^{s_i'}( A_0)$. Thus, as in the proof of Lemma \ref{connection estimate},
\be
\label{morse}
||A_{i}-A_0||_{L^2(E)}\leq C ||e^{-s'_i}\partial_{A_0}e^{2s'_i}||_{L^2(E)}\leq C||F_{A_i}||_{L^2(E)}\rightarrow 0.
\ee
 Furthermore, since $A_0$ is flat, we can integrate by parts and change the order of derivatives to conclude:
\bea
||\nabla^0(A_i-A_0)||_{L^2(E)}^2&=&\int_{E_p}{\rm Tr}(\nabla^0(e^{-2s_i} \nabla^0 e^{2s_i})\left(\nabla^0(e^{-2s_i}\nabla^0 e^{2s_i})\right)^*)\nonumber\\
&=&\int_{E_p}{\rm Tr}(\bar\nabla^0(e^{-2s_i} \nabla^0 e^{2s_i})\left(\bar\nabla^0(e^{-2s_i}\nabla^0 e^{2s_i})\right)^*)\nonumber\\
&=&||F_{A_i}||^2_{L^2(E)}\rightarrow0.\nonumber
\eea
Thus we have demonstrated 
\be
||A_{i}-A_0||_{L^2_1(E)}\rightarrow 0,\nonumber 
\ee
which is the stated convergence in  Theorem \ref{maintheorem}.

Next we prove smooth convergence, allowing for the action of unitary gauge transformations. Specifically, since $x$ is away from the bubbling set, there exists a small disk $D_\rho(x)$ that does not intersect $W_1\cup W_2$. We use the same coordinate transformations $\lambda_i$ in the proof Proposition \ref{bubbling} sending $x$ to the origin and scaling. Consider the rescaled metrics $\ti\omega_i=t_i^{-2}\lambda_i^*\omega_{t_i}$. The set $\pi^{-1}(D_\rho(x))$ rescales to a set topologically equivalent to $\ti D_{\frac{\rho}{{ t_i}}}(0)\times E$ (although the complex structure will not be a product). 

Our sequence of connections $\Xi_i$ pulls back to $\ti \Xi_i$, with fiber and base components 
\be
\ti A^i_j=A^i_j\qquad{\rm and}\qquad \ti B^i_j= { t_i} B^i_j.\nonumber
\ee
As before, the connections $\ti\Xi_i$ are HYM with respect to $\ti\omega_{ i}$, and each $\ti\omega_i$ is uniformly equivalent to the Euclidian metric for large $i$ by Proposition \ref{HTbound}. Also, since $D_\rho(x)$ is away from $W_1$, the function $ t_i^2m_i(x)$ is uniformly bounded above on the disk. This implies the curvature  $|F_{\ti \Xi_i}|$ is uniformly bounded on $\pi^{-1}(\ti D_{\frac{\rho}{ { t_i}}}(0))$.

Applying strong Uhlenbeck compactness  \cite[Corollary 1.4 and Theorem 1.5]{U2} on the fixed compact set 
\be
\pi^{-1}(\ti D_{1}(0))\subset \pi^{-1}(\ti D_{\frac{\rho}{{ t_i}}}(0))\nonumber,
\ee
 there exists a sequence of gauge transformations $u_i$ so that along a subsequence, $u_i^* \ti\Xi_i$ converges smoothly  to a limiting Yang-Mills connection $\ti \Xi_\infty$. Restricting our attention to the fiber  $E$ over the origin yields a sequence of connections $u_i^*\ti A_i$ which converges smoothly to a limiting flat  connection $\ti A_\infty$. Note that our fiber coordinates are not scaled, and the restriction of $\ti\omega_{ t_i}$ to $E$ is equivalent the standard metric $\omega_0=dy^1\wedge dy^2$. Thus, on $E$, we see $u_i^*A_i$ converges smoothly to a flat connection $A_\infty$. The connection $A_\infty$ may not equal $A_0$, but it will lie in the unitary gauge orbit.
 
We conclude by remarking that estimates of the form \eqref{morse} are common in these types of degeneration problems, for example see Proposition 3.1 in \cite{Fuk2} or Theorem 1 in \cite{F}. In our   estimate,   the curvature term is not raised to a power, and this holds because the specific form of our complex structure $V|_E= \oplus_{j=1}^n\mathcal O_E(q_j-0)$ implies that the Yang-Mills energy functional is Morse-Bott at $A_0$ (see Definition 7.5 in \cite{F}). Essentially,  the argument in our proof of Proposition \ref{poincare} gives that the kernel of the Hessian operator of the Yang-Mills energy functional can be identified with one forms valued in constant diagonal matrices, which also gives the tangent space to Yang-Mills connections at $A_0$. We direct the reader to \cite{F} for further details.

\newpage

\end{normalsize}

\medskip
\medskip
 
\end{document}